\documentclass[12pt,oneside,english]{amsart}
\usepackage[T1]{fontenc}
\usepackage[latin9]{inputenc}
\usepackage{geometry}
\usepackage{babel}
\usepackage{amstext}
\usepackage{amsthm}
\usepackage{amssymb}
\usepackage{color}
\usepackage{esint}
\usepackage[unicode=true,pdfusetitle,
 bookmarks=true,bookmarksnumbered=false,bookmarksopen=false,
 breaklinks=false,pdfborder={0 0 1},backref=false,colorlinks=false]
 {hyperref}

\makeatletter
\numberwithin{equation}{section}
\numberwithin{figure}{section}
\theoremstyle{plain}
\newtheorem{thm}{\protect\theoremname}
\theoremstyle{plain}
\newtheorem{cor}[thm]{\protect\corollaryname}
\theoremstyle{plain}
\newtheorem{lem}[thm]{\protect\lemmaname}
\theoremstyle{plain}
\newtheorem{conjecture}[thm]{\protect\conjecturename}
\theoremstyle{definition}

\theoremstyle{plain}

\theoremstyle{remark}

\makeatother

\providecommand{\conjecturename}{Conjecture}
\providecommand{\corollaryname}{Corollary}
\providecommand{\definitionname}{Definition}
\providecommand{\lemmaname}{Lemma}
\providecommand{\propositionname}{Proposition}
\providecommand{\remarkname}{Remark}
\providecommand{\theoremname}{Theorem}
\newcommand{\sech}{\operatorname{sech}}

\begin{document}
\title{on the discrete convolution of the liouville and m\"obius functions}
\author{marco cantarini, alessandro gambini, alessandro zaccagnini}
\date{\today}
\begin{abstract}
In this article we study some properties of the discrete convolution
of Liouville function $S(n):=\sum_{m_{1}+m_{2}=n}\lambda\left(m_{1}\right)\lambda\left(m_{2}\right)$,
which is a Goldbach-type counting function of representations. In
particular, using the general approach introduced in a recent paper
\cite{CGZ}, we will give an explicit formula for weighted averages
of $S(n)$ with a general weights $f(w)$ that verify suitable conditions.
This formula allows us to obtain interesting information about the
Dirichlet and power series of $S(n)$ and the discrete convolution
with an arbitrary numbers of factors $\lambda(n)$.
\end{abstract}

\maketitle

\section{introduction}

In number theory, in particular in the study of properties of primes,
one of the main problems is to control, in a suitable sense, functions
with a very erratic behavior. A classical, and very useful, tool is
to study the so called explicit formulas of the average, weighted
or not, of such functions. The idea is to exploit the connection
of the problem with some properties of important special functions,
like the Riemann zeta function $\zeta\left(s\right)$. A very well-known
example is Goldbach's conjecture which, essentially, is the problem
to show that
\[
R(N):=\sum_{n\leq N}\Lambda\left(n\right)\Lambda\left(N-n\right)>0
\]
for every even number $N>2$, that is, the discrete convolution of
the von Mangoldt function
\[
\Lambda\left(n\right):=\begin{cases}
\log\left(p\right), & n=p^{m},\,p\text{ prime, }m\in\mathbb{N}\\
0, & \text{otherwise,}
\end{cases}
\]
with itself is positive for every even integer greater than $2$. For the study
of averages, weighted or not, of $R(N)$ and other similar function, in different settings, the interested reader may see, e.g., \cite{BKP2019, CGZ2022, CGZ2021, GS2023, GY2017, LZ2015}.

The Liouville function $\lambda\left(n\right)$, defined as
\[
\lambda(n):=\begin{cases}
1, & n=1\text{ or }n\text{ is the product of an even number of prime numbers,}\\
-1, & n\text{ is the product of an odd number of prime numbers,}
\end{cases}
\]
is surely one on the most studied arithmetical functions in recent
years, that attracted many number theory experts like %of the caliber of
Tao, Matom\"aki, Radziwi\l \l{} (see, for example, \cite{AHZ,MMT,MR1,MR2,MRT}),
just to recall a few important names. 

A very interesting problem linked to such function is surely the so-called
Chowla conjecture, which states that for all $h\geq1$ one has
\[
\frac{1}{h}\left|\sum_{n\leq x}\lambda\left(n\right)\lambda\left(n+h\right)\right|=o\left(x\right)
\]
as $x\rightarrow+\infty$ (see \cite{C}), which is morally linked,
when $h=2$, to the twin-prime asymptotic conjecture
\begin{equation}
\frac{1}{2}\sum_{n\leq x}\Lambda\left(n\right)\Lambda\left(n+2\right)=x\Pi_{2}+o\left(x\right)\label{eq: twin primes}
\end{equation}
as $x\rightarrow+\infty$, where $\Pi_{2}$ is the twin-prime constant. It is well-known
that (\ref{eq: twin primes}) is strictly linked to the Goldbach
conjecture in the form
\[
R(N)=N\mathfrak{S}\left(N\right)+o\left(N\right)
\]
as $N\rightarrow+\infty,$ where $\mathfrak{S}\left(N\right)$ is
the so-called singular series. Indeed, the two problems have, essentially,
similar difficulties, which have not yet allowed number theorist to
find a rigorous proof. From this observation, it was quite natural
to focus the attention on the convolution sum
\begin{equation}
\label{def-S}
S\left(N\right):=\sum_{n\leq N}\lambda\left(n\right)\lambda\left(N-n\right).
\end{equation}
In \cite{CK} Corr\'adi and K\'atai, from some heuristic ideas,
conjectured that $S\left(N\right)=o\left(N\right)$ as $N\rightarrow+\infty$
and this was proved in \cite{DkGK} under the assumption that there
are infinitely many Siegel zeros. In 2018 AIM workshop on Sarnak\textquoteright s
conjecture \cite{AIM} was posed the following conjecture:
\[
\left|S\left(N\right)\right|<N-1
\]
for every $N\geq2$. This was recently solved by Mangerel in \cite{M} for $N$
sufficiently large.

Motivated by these results, we asked ourselves if it was possible
to find explicit formulas for weighted averages regarding the function
$S\left(n\right)$, in line with what was done for the function that
counts Goldbach numbers $R\left(N\right)$ and the answer is positive,
under reasonable assumptions. 

Let
\begin{equation}
\label{def-L}
L\left(x\right):=\sum_{n\leq x}\lambda\left(n\right) \qquad \text{and} \qquad C_\lambda(x) := \sum_{n\leq x}S\left(n\right)\left(x-n\right).
\end{equation}
Indeed, in this paper we show a truncated explicit formula for the
Laplace convolution of $L\left(x\right)$
with itself, where such convolution involving $L\left(x\right)$ corresponds to $C_\lambda(x)$.

Then, taking advantage of the general results proven in \cite{CGZ},
we are able to find a truncated explicit formula for a general weighted
average of $S\left(n\right)$, where the weight verifies some reasonable
hypotheses. From this general formula, we provide, as examples, some
explicit formulas for the Dirichlet and power series $\sum_{n\geq1}S\left(n\right)n^{-s},\,\sum_{n\geq1}S\left(n\right)e^{-ny}$
in a very easy way. We also prove that $\sum_{n\geq1}S\left(n\right)n^{-s}$
can be analytically continued to the half-plane $\text{Re}(s)>1.$
%\textcolor{red}{Dare definizioni}

Then, we show that a general formula can be proved also for the weighted
averages of the $d-$term convolution function
\[
S_{d}\left(N\right):=\sum_{m_{1}+\dots+m_{d}=N}\lambda\left(m_{1}\right)\cdots\lambda\left(m_{d}\right),\,d\geq2.
\]
%In the last part, we show that, using some tools from the theory of
%distributions, the obtained results allow us to prove the non-trivial bound
%\[
%\sum_{n\leq x}S\left(n\right)=O\left(x\right)
%\]
%as $x\rightarrow+\infty$ which, in some sense, corroborates the reasonableness
%of the conjecture Corr\'adi and K\'atai. 

What we said about $\lambda(n)$ can be, essentially, derived in same
way for the M\"obius function $\mu(n)$. It is known that the Goldbach-type
problem for the function $\mu(n)$, that is, the evaluation of the
sum
\[
S^{*}(n):=\sum_{m_{1}+m_{2}=n}\mu\left(m_{1}\right)\mu\left(m_{2}\right)
\]
is linked to the function $R(n)$, due to the Dirichlet product;
that is, from
\[
\Lambda\left(n\right)=\sum_{d\vert n}\log\left(\frac{n}{d}\right)\mu\left(d\right)
\]
which means that a ``good control'' of
\[
\sum_{N_{1}m_{1}+N_{2}m_{2}=n}\mu\left(m_{1}\right)\mu\left(m_{2}\right)
\]
where $N_{1},N_{2}$ are positive integers, implies information about
$R(n)$. 

More precisely, one of our main results is the following.
\begin{thm}
Let $\eta>0$ and let $f:\mathbb{R}\rightarrow\mathbb{C}$. Assume
that:

1) $f$ has its support in $\left[a,b\right),\,a<b$, $a\in\mathbb{R},\,\eta a<1$
and $b\in\mathbb{R}\cup\left\{ +\infty\right\} $.

2) $f\in C^{1}\left(a,b\right)$.

3) $f^{\prime}$ is absolutely continuous in $\left(a,b\right)$.

4) $f\left(a^{+}\right),f^{\prime}\left(a^{+}\right)$ exist and are
finite, and $f\left(b^{-}\right)=f^{\prime}\left(b^{-}\right)=0$.

Then, assuming the
Riemann hypothesis and the simplicity of the non-trivial zeros of
$\zeta\left(s\right)$, we get
\[
\sum_{n\leq\eta b}\sum_{m\leq\eta b-n}\lambda\left(n\right)\lambda\left(m\right)f\left(\frac{m+n}{\eta}\right)=\frac{\pi\eta}{8\zeta\left(\frac{1}{2}\right)^{2}}\int_{a}^{b}f^{\prime\prime}\left(w\right)w^{2}dw
\]
\[
+\frac{\sqrt{\pi}}{\zeta\left(\frac{1}{2}\right)}\sum_{\rho}\frac{\zeta\left(2\rho\right)\Gamma\left(\rho\right)\eta^{\rho+1/2}}{\zeta^{\prime}\left(\rho\right)\Gamma\left(\rho+\frac{5}{2}\right)}\int_{a}^{b}f^{\prime\prime}\left(w\right)w^{\rho+1/2+1}dw
\]
\[
+\sum_{\rho_{1}}\frac{\zeta\left(2\rho_{1}\right)}{\zeta^{\prime}\left(\rho_{1}\right)}\sum_{\rho_{2}}\frac{\zeta\left(2\rho_{2}\right)\eta^{\rho_{1}+\rho_{2}}}{\zeta^{\prime}\left(\rho_{2}\right)}\frac{\Gamma\left(\rho_{1}\right)\Gamma\left(\rho_{2}\right)}{\Gamma\left(\rho_{1}+\rho_{2}+2\right)}\int_{a}^{b}f^{\prime\prime}\left(w\right)w^{\rho_{1}+\rho_{2}+1}dw
\]
\[
+O_{\varepsilon}\left(\eta^{1/2+\varepsilon}\int_{a}^{b}w^{3/2+\varepsilon}\left|f^{\prime\prime}\left(w\right)\right|dw+\int_{a}^{b}w\left|f^{\prime\prime}\left(w\right)\right|dw\right).
\]
In the case $a\eta \geq 1$, in the previous formula we have the extra
term
\[
L\left(\eta a\right)\int_{a}^{b}L\left(\eta v-\eta a\right)f^{\prime}\left(v\right)dv.
\]
\end{thm}

%\textcolor{red}{Nella formula precedente per ipotesi $a\eta=1$: si pu\`o semplificare?}.
From this results, as a simple application, we derive what follows.
\begin{cor}
Assume the Riemann hypothesis and the simplicity of the non-trivial
zeros of $\zeta\left(s\right)$. Then, the Dirichlet series $\sum_{n\geq1}S\left(n\right)n^{-s}$
admits an analytic continuation up to $\text{Re}(s)>1$. If we also
assume that the imaginary parts of the non-trivial zeros are linearly
independent over the rational numbers, then the line $\text{Re}(s)=1$ is
the natural boundary of the series.
\end{cor}

%Concerning a non-trivial bound for the average of $S\left(n\right)$,
%we get what follows.
%\begin{thm}
%Assume the Riemann hypothesis and the simplicity of the non-trivial
%zeros of $\zeta\left(s\right)$. Then
%\[
%\sum_{n\leq x}S\left(n\right)=O\left(x\right)
%\]
%as $x\rightarrow+\infty$.
%\end{thm}

The same kind of results can be reached for the function $S^{*}(n)$,
with the same ideas.

\section{preliminary results}

In this part we present some preliminary results. The idea to prove
the main theorem is to exploit Proposition 2 of \cite{CGZ} but, as
we know from the same paper, the starting point is to find
the explicit formula of
\[
  C_\lambda(x)
  =
  \sum_{n\leq x}S\left(n\right)\left(x-n\right)=\int_{0}^{x}L(y)L(x-y)dy
\]
where $L(y)$ is defined in \eqref{def-L}, and the same
holds for $M(y):=\sum_{n\leq y}\mu(n)$. So, the first part of this
section is to prove some results for this aim. The idea is to insert
the explicit formula of $L\left(y\right)$ in the integral and then
integrating termwise, but it is evident that in such type of approach
the contribution when $y$ is close to $0$ is important as much as
that when $y$ is near $x$. In order to obtain combination of series
over the non-trivial zeros with tractable special functions (essentially
the Gamma and Beta function), it is important to study the explicit
formula of $L\left(y\right)$ not only in the classical range $y\geq1$
but also for $y>0$. Luckily, for what concerns $L\left(y\right)$, this
extension is possible and does not involve changes in the law, as
happens with $\psi\left(x\right):=\sum_{n\leq x}\Lambda\left(n\right).$
The same argument and considerations hold for $M(y)$.

With this idea in mind, we need to recall the following results.
\begin{lem}
(Perron's formula).

Let $\sigma_{a}$ be the abscissa of absolute convergence of the Dirichlet
series 
\[
\alpha(s):=\sum_{n\geq1}a_{n}n^{-s}
\]
 and $x>0,\,T\geq1$. Then, if $\text{\ensuremath{\sigma_{0}>\max\left(\sigma_{a},0\right)}}$,
we have
\[
\sum_{n\leq x}a_{n}=\frac{1}{2\pi i}\int_{\sigma_{0}-iT}^{\sigma_{0}+iT}\alpha\left(s\right)\frac{x^{s}}{s}ds+E\left(x\right)+R\left(x,T\right)
\]
where 
\[
E\left(x\right)=\begin{cases}
a_{x}/2, & x\in\mathbb{N}\\
0, & \text{otherwise}
\end{cases}
\]
and
\[
R\left(x,T\right)\ll\sum_{\underset{{\scriptstyle n\neq x}}{x/2<n<2x}}\left|a_{n}\right|\min\left(1,\frac{x}{T\left|x-n\right|}\right)+\frac{4^{\sigma_{0}}+x^{\sigma_{0}}}{T}\sum_{n\geq1}\frac{\left|a_{n}\right|}{n^{\sigma_{0}}}.
\]
\end{lem}

For a reference, see Theorem $5.2$ and Corollary $5.3$ of \cite{MV}.
We underline that the previous result holds for every $x>0$.

Before starting considering the study the convolution involving $L(x)$,
we introduce a technical theorem that will be useful later.
\begin{thm}
\label{thm:double ser}Assume the Riemann hypothesis (RH), let $\rho$
run over the non-trivial zeros of $\zeta(s)$ and let $f(z)$ be a
complex function defined on the set of the
non-trivial zeros and such that
\[
\sum_{0<\gamma<T}\left|\frac{f\left(\rho\right)}{\rho}\right|=o\left(T^{\alpha}\right),\sum_{0<\gamma<T}\left|\frac{f\left(\overline{\rho}\right)}{\overline{\rho}}\right|=o\left(T^{\alpha}\right)
\]
as $T\rightarrow+\infty,$ for every $\alpha>0.$ Then the double
series
\[
\sum_{\rho_{1}}f\left(\rho_{1}\right)\sum_{\rho_{2}}f\left(\rho_{2}\right)\frac{\Gamma\left(\rho_{1}\right)\Gamma\left(\rho_{2}\right)}{\Gamma\left(\rho_{1}+\rho_{2}+1+k\right)}
\]
converges absolutely if $k\in\mathbb{R},$ $k>1/2$.
\end{thm}

\begin{proof}
We want to prove that
\begin{equation}
\label{double-sum}
\sum_{\rho_{1}}\left|f\left(\rho_{1}\right)\right|\sum_{\rho_{2}}\left|f\left(\rho_{2}\right)\right|\left|\frac{\Gamma\left(\rho_{1}\right)\Gamma\left(\rho_{2}\right)}{\Gamma\left(\rho_{1}+\rho_{2}+1+k\right)}\right|
\end{equation}
is finite.
By the symmetries of the zeros, it is enough to consider the cases
$\gamma_{1}>0,\gamma_{2}>0$ and $\gamma_{1}>0,\gamma_{2}<0.$ Let
us consider the first configuration. From Stirling's formula
\[
\left|\Gamma\left(x+iy\right)\right|\sim\sqrt{2\pi}e^{-\pi\left|y\right|/2}\left|y\right|^{x-1/2},
\]
where $x\in\left[x_{1},x_{2}\right]$ and $\left|y\right|\rightarrow+\infty$,
we have that the contribution of these zeros to \eqref{double-sum} is
\[
  \ll
\sum_{\rho_{1}:\gamma_{1}>0}\left|f\left(\rho_{1}\right)\right|\sum_{\rho_{2}:\gamma_{2}>0}\left|f\left(\rho_{2}\right)\right|\frac{1}{\left(\gamma_{1}+\gamma_{2}\right)^{3/2+k}}
\]
\[
\ll\sum_{\rho_{1}:\gamma_{1}>0}\left|f\left(\rho_{1}\right)\right|\frac{1}{\gamma_{1}^{3/4+k/2}}\sum_{\rho_{2}:\gamma_{2}>0}\left|f\left(\rho_{2}\right)\right|\frac{1}{\gamma_{2}^{3/4+k/2}}
\]
\[
\ll\sum_{\rho_{1}:\gamma_{1}>0}\left|\frac{f\left(\rho_{1}\right)}{\rho_{1}}\right|\frac{1}{\gamma_{1}^{k/2-1/4}}\sum_{\rho_{2}:\gamma_{2}>0}\left|\frac{f\left(\rho_{2}\right)}{\rho_{2}}\right|\frac{1}{\gamma_{2}^{k/2-1/4}}
\]
by the arithmetic-geometric inequality and the series converges absolutely if $k>1/2$ since, by partial
summation
\[
\sum_{\rho:0<\gamma<T}\left|\frac{f\left(\rho\right)}{\rho}\right|\frac{1}{\gamma^{k/2-1/4}}=\sum_{\rho:0<\gamma<T}\left|\frac{f\left(\rho\right)}{\rho}\right|\frac{1}{T^{k/2-1/4}}-\frac{2k-1}{4}\int_{14}^{T}\sum_{\rho:0<\gamma<t}\left|\frac{f\left(\rho\right)}{\rho}\right|t^{-k/2-3/4}dt.
\]
The convergence follows from our hypothesis on $f$. Now let us consider the case $\gamma_{1}>0,\,\gamma_{2}<0$.
Then we have to deal with
\[
\sum_{\rho_{1}:\gamma_{1}>0}\left|f\left(\rho_{1}\right)\right|\sum_{\rho_{2}:\gamma_{2}>0}\left|f\left(\overline{\rho_{2}}\right)\right|\left|\frac{\Gamma\left(\rho_{1}\right)\Gamma\left(\overline{\rho_{2}}\right)}{\Gamma\left(\rho_{1}+\overline{\rho_{2}}+1+k\right)}\right|.
\]
Let $0<\delta<1$ be fixed. We split the second series in the following
way
\[
\sum_{\rho_{2}:\gamma_{2}>0}=\sum_{\underset{{\scriptstyle \left|\gamma_{1}-\gamma_{2}\right|<\delta\max\left(\gamma_{1},\gamma_{2}\right)}}{\rho_{2}:\gamma_{2}>0}}+\sum_{\underset{{\scriptstyle \left|\gamma_{1}-\gamma_{2}\right|\geq\delta\max\left(\gamma_{1},\gamma_{2}\right)}}{\rho_{2}:\gamma_{2}>0}}
\]
\[
=:\sideset{}{_{1}}\sum+\sideset{}{_{2}}\sum.
\]
 We start with $\sideset{}{_{1}}\sum$. Recalling the relation
\[
\left|\Gamma\left(x+iy\right)\right|\geq\Gamma\left(x\right)
  \bigl(\sech\left(\pi y\right)\bigr)^{1/2},\,x\geq1/2,
\]
(see \cite{O}, formula $5.6.7$) we observe that
%\footnote{Concordo con il risultato finale, ma non riesco a seguire i dettagli (AZ)}
\[
\left|\Gamma\left(\rho_{1}+\overline{\rho_{2}}+1+k\right)\right|\geq\min_{-\delta\max\left(\gamma_{1},\gamma_{2}\right)\leq t\leq\delta\max\left(\gamma_{1},\gamma_{2}\right)}\left|\Gamma\left(2+k+it\right)\right|
\]
\[
\geq\Gamma\left(2+k\right)\min_{0\leq t\leq\delta\max\left(\gamma_{1},\gamma_{2}\right)}
  \bigl(\sech\left(\pi t\right)\bigr)^{1/2}\gg_{k}e^{-\delta\pi\left(\gamma_{1}+\gamma_{2}\right)/2}
\]
so, using Stirling's formula in the Gamma functions at the numerator, we get
\[
\sum_{\rho_{1}:\gamma_{1}>0}\left|f\left(\rho_{1}\right)\right|\sideset{}{_{1}}\sum\left|f\left(\overline{\rho_{2}}\right)\right|\left|\frac{\Gamma\left(\rho_{1}\right)\Gamma\left(\overline{\rho_{2}}\right)}{\Gamma\left(\rho_{1}+\overline{\rho_{2}}+1+k\right)}\right|
\]
\[
\ll_{k}\sum_{\rho_{1}:\gamma_{1}>0}\left|f\left(\rho_{1}\right)\right|e^{-\pi\left(1-\delta\right)\gamma_{1}/2}\sum_{\rho_{2}:\gamma_{2}>0}\left|f\left(\overline{\rho_{2}}\right)\right|e^{-\pi\left(1-\delta\right)\gamma_{2}/2}
\]
and the convergence follows trivially. Now we consider $\sideset{}{_{2}}\sum.$
If $\left|\gamma_{1}-\gamma_{2}\right|\geq\delta\max\left(\gamma_{1},\gamma_{2}\right)>0$,
then $\gamma_{1}\neq\gamma_{2}$. We split $\sideset{}{_{2}}\sum$
in two further cases: if $\gamma_{1}>\gamma_{2}$ we have, by Stirling's formula again, that 
\[
\sum_{\rho_{1}:\gamma_{1}>0}\left|f\left(\rho_{1}\right)\right|\sum_{\rho_{2}:\gamma_{1}-\gamma_{2}\geq\delta\gamma_{1}}\left|f\left(\overline{\rho_{2}}\right)\right|\frac{e^{-\pi\gamma_{2}/2}e^{-\pi\gamma_{1}/2}}{\left(\gamma_{1}-\gamma_{2}\right)^{3/2+k}e^{-\pi\left(\gamma_{1}-\gamma_{2}\right)/2}}
\]
\[
\ll_{\delta}\sum_{\rho_{1}:\gamma_{1}>0}\left|f\left(\rho_{1}\right)\right|\frac{1}{\gamma_{1}^{3/2+k}}\sum_{\rho_{2}:\gamma_{2}>0}\left|f\left(\overline{\rho_{2}}\right)\right|e^{-\pi\gamma_{2}}
\]
and again the convergence follows. If $\gamma_{1}<\gamma_{2}$ we
get, similarly to the previous case
\[
\sum_{\rho_{1}:\gamma_{1}>0}\left|f\left(\rho_{1}\right)\right|\sum_{\rho_{2}:\gamma_{2}-\gamma_{1}\geq\delta\gamma_{2}}\left|f\left(\overline{\rho_{2}}\right)\right|\frac{e^{-\pi\gamma_{2}/2}e^{-\pi\gamma_{1}/2}}{\left(\gamma_{2}-\gamma_{1}\right)^{3/2+k}e^{-\pi\left(\gamma_{2}-\gamma_{1}\right)/2}}
\]
\[
\ll_{\delta}\sum_{\rho_{1}:\gamma_{1}>0}\left|f\left(\rho_{1}\right)\right|e^{-\pi\gamma_{1}}\sum_{\rho_{2}:\gamma_{2}>0}\left|f\left(\overline{\rho_{2}}\right)\right|\frac{1}{\gamma_{2}^{3/2+k}}
\]
and then the convergence of the series.
\end{proof}
Note that, probably, the assumption of RH can be avoided, but it is
not important for our aims. Now, we need some results about the explicit
formula for $L(y)$ for $y>0$. In the literature we can find results
that hold only for $y\geq1$, so we need to extend such results also
in the case $0<y<1$. We start with the following corollary.
\begin{cor}
For $x>0$, $T\geq1$ and 
\begin{equation}
\sigma_{0}=\sigma_{0}\left(x\right):=\begin{cases}
1+1/\left|\log\left(x\right)\right|, & x>0,\,x\neq1\\
2, & x=1
\end{cases}\label{eq:def sigma0}
\end{equation}
 we have that 
\[
L\left(x\right)=\frac{1}{2\pi i}\int_{\sigma_{0}-iT}^{\sigma_{0}+iT}\frac{\zeta\left(2s\right)}{\zeta\left(s\right)}\frac{x^{s}}{s}ds+E\left(x\right)+R\left(x,T\right)
\]
where 
\[
E(x):=\begin{cases}
\lambda\left(x\right)/2, & x\in\mathbb{N}\\
0, & \text{otherwise}
\end{cases}
\]
and
\[
R\left(x,T\right)\ll1+\frac{x\left(\left|\log\left(x\right)\right|+1\right)}{T}.
\]
\end{cor}

\begin{proof}
The case $x\geq1$ can be found in Corollary $4.2$ of \cite{H}, so
we need to prove the result for $0<x<1$. Actually the proof is essentially
the same, we just need to change some passages. We take $\sigma_{0}>1$,
so by Perron's formula above we get
\[
L\left(x\right)=\frac{1}{2\pi i}\int_{\sigma_{0}-iT}^{\sigma_{0}+iT}\frac{\zeta\left(2s\right)}{\zeta\left(s\right)}\frac{x^{s}}{s}ds+E\left(x\right)+R\left(x,T\right).
\]
Now we have to estimate $R\left(x,T\right)$. Note that, in this case,
we trivially have
\[
\sum_{\underset{{\scriptstyle n\neq x}}{x/2<n<2x}}\min\left(1,\frac{x}{T\left|x-n\right|}\right)\leq\sum_{n\leq2}1\ll1
\]
and the second term is
\[
\frac{x^{\sigma_{0}}}{T}\sum_{n\geq1}\frac{1}{n^{\sigma_{0}}}=\frac{x^{\sigma_{0}}}{T}\zeta\left(\sigma_{0}\right)
\ll\frac{x^{\sigma_{0}}}{T\left(\sigma_{0}-1\right)}.
\]
If $x\geq1$ we have the same situation of \cite{H} and so, combining
all the results, we get
\[
R\left(x,T\right)\ll1+\frac{x \vert \log\left(x\right) \vert}{T}+\frac{x^{\sigma_{0}}}{T\left(\sigma_{0}-1\right)}.
\]
Taking $\sigma_{0}$ as in (\ref{eq:def sigma0}) we have the thesis.
\end{proof}
It is quite simple to observe that, if we replace $\zeta\left(2s\right)$
in the integral with the constant function equal to one, that is,
we work with the M\"obius function instead of the Liouville function,
we get the same results. So we have
\begin{cor}
For $x>0$, $T\geq1$ and 
\begin{equation}
\sigma_{0}=\sigma_{0}\left(x\right):=\begin{cases}
1+1/\left|\log\left(x\right)\right|, & x>0,\,x\neq1\\
2, & x=1
\end{cases}\label{eq:def sigma0-1}
\end{equation}
 we have that 
\[
M\left(x\right)=\frac{1}{2\pi i}\int_{\sigma_{0}-iT}^{\sigma_{0}+iT}\frac{1}{\zeta\left(s\right)}\frac{x^{s}}{s}ds+E^{*}\left(x\right)+R^{*}\left(x,T\right)
\]
where 
\[
E^{*}(x):=\begin{cases}
\mu\left(x\right)/2, & x\in\mathbb{N}\\
0, & \text{otherwise}
\end{cases}
\]
and
\[
R^{*}\left(x,T\right)\ll1+\frac{x\left(\left|\log\left(x\right)\right|+1\right)}{T}.
\]
\end{cor}

For the next results we need the following assumption that can be
found in \cite{H}, but actually present in similar forms in many
other papers (see, e.g., \cite{G}).
\begin{conjecture}
\label{conj:simple zeros}(SZ) We have that
\[
\sum_{0<\gamma<T}\frac{1}{\left|\zeta^{\prime}\left(\rho\right)\right|}\ll T(\log T)^{1/2}.
\]
\end{conjecture}

It is well-known that Conjecture \ref{conj:simple zeros} implies
that all non-trivial zeros are simple and also the bounds
\begin{align*}
  \sum_{0<\gamma<T}\frac{\left|\zeta\left(2\rho\right)\right|}{\left|\rho\zeta^{\prime}\left(\rho\right)\right|}
  &\ll
  (\log T)^{3/2}\log\left(\log\left(T\right)\right) \\
  \sum_{0<\gamma<T}\frac{1}{\left|\rho\zeta^{\prime}\left(\rho\right)\right|}
  &\ll
  (\log T)^{3/2}
\end{align*}
(see, e.g., \cite{H}, \cite{M-1}, \cite{N}). We recall the following
estimations.
\begin{lem}
Assume RH and $\left|t\right|\geq1$. Then for all $\varepsilon>0$
we have
\begin{equation}
\zeta\left(\sigma+it\right)\ll\begin{cases}
t^{1/2-\sigma+\varepsilon}, & 0<\sigma<1/2\\
t^{\varepsilon}, & \sigma\geq1/2,
\end{cases}\label{eq:first est}
\end{equation}
where the implicit constant depends only on $\varepsilon$ and for
fixed $\delta>0$ we have
\begin{equation}
\frac{1}{\zeta\left(\sigma+it\right)}\ll\begin{cases}
t^{-1/2+\sigma+\varepsilon}, & 0<\sigma\leq1/2-\delta\\
t^{\varepsilon}, & \sigma\geq1/2+\delta,
\end{cases}\label{eq:second est}
\end{equation}
where the implicit constant depends only on $\varepsilon$ and $\delta$.
\end{lem}

See \cite{MV}, Corollary 10.5, Theorem 13.18, Theorem 13.23 or \cite{H}.
\begin{lem}
Assume RH and the simplicity of the non-trivial zeros of $\zeta\left(s\right)$.
%Let $z=x+iy$, with $x>0$ and $y\in\mathbb{R}$. 
Then, there exists a sequence $\mathcal{T}=\left(T_{\nu}\right)_{\nu\in\mathbb{N}}$
with $\nu\leq T_{\nu}\leq\nu+1$ such that for all $\varepsilon>0$
and $0<\sigma\leq2$ we have.
\begin{equation}
\frac{1}{\zeta\left(\sigma+iT_{\nu}\right)}\ll_{\varepsilon}T_{\nu}^{\varepsilon}.\label{eq:special est}
\end{equation}
\end{lem}

See \cite{MV}, Theorem 13.23 or again \cite{H}. Now, introducing the notation $\int_{\left(a\right)}:=\int_{a-i\infty}^{a+i\infty}$, we are able to
prove the following results.
\begin{thm}
Assume RH and SZ. Then for $T_{\nu}\in\mathcal{T}$ and $x>0$ we
get
\[
L(x)=\frac{x^{1/2}}{\zeta\left(\frac{1}{2}\right)}+\sum_{\left|\gamma\right|<T_{\nu}}\frac{\zeta\left(2\rho\right)x^{\rho}}{\zeta^{\prime}\left(\rho\right)\rho}+E(x)+I(x)+R\left(x,T_\nu\right)
\]
where
\[
E(x):=\begin{cases}
\lambda\left(x\right)/2, & x\in\mathbb{N}\\
0, & \text{otherwise,}
\end{cases}
\]

\[
I\left(x\right):=\frac{1}{2\pi i}\int_{\left(\varepsilon\right)}\frac{\zeta\left(2s\right)}{\zeta\left(s\right)}\frac{x^{s}}{s}ds
\]
and
\[
R\left(x,T_\nu\right)\ll_{\varepsilon}1+\frac{x\left(\left|\log\left(x\right)\right|+1\right)}{T_{\nu}}+\frac{x-x^{1/4}}{T_{\nu}^{1-\varepsilon}\log\left(x\right)}
\]
for arbitrarily small $0<\varepsilon<1/4$, where the indeterminate form at $x=1$ must be interpreted as a limit.
\end{thm}

\begin{proof}
The proof is again similar to the one present in Theorem $4.5$ of
\cite{H} but again we need to evaluate also the case $0<x<1$. We
have the following decomposition, for $T \geq 1$ and with $T$ as an element of $\mathcal{T},$ that is, $T=T_{\nu}$ for some $\nu\in\mathbb{N}.$
\[
\frac{1}{2\pi i}\int_{\sigma_{0}-iT}^{\sigma_{0}+iT}\frac{\zeta\left(2s\right)}{\zeta\left(s\right)}\frac{x^{s}}{s}ds=\frac{1}{2\pi i}\oint_{C_{T}}\frac{\zeta\left(2s\right)}{\zeta\left(s\right)}\frac{x^{s}}{s}ds
\]
\[
-\frac{1}{2\pi i}\int_{\varepsilon-iT}^{\varepsilon+iT}\frac{\zeta\left(2s\right)}{\zeta\left(s\right)}\frac{x^{s}}{s}ds-\frac{1}{2\pi i}\left(\int_{\sigma_{0}+iT}^{\varepsilon+iT}+\int_{\varepsilon-iT}^{\sigma_{0}-iT}\right)\frac{\zeta\left(2s\right)}{\zeta\left(s\right)}\frac{x^{s}}{s}ds
\]
where $C_{T}$ is the rectangle with vertices in $\sigma_{0}\pm iT,\varepsilon\pm iT$.
Clearly
\[
\frac{1}{2\pi i}\oint_{C_{T}}\frac{\zeta\left(2s\right)}{\zeta\left(s\right)}\frac{x^{s}}{s}ds=\frac{x^{1/2}}{\zeta\left(\frac{1}{2}\right)}+\sum_{\left|\gamma\right|<T}\frac{\zeta\left(2\rho\right)x^{\rho}}{\zeta^{\prime}\left(\rho\right)\rho}
\]
from the residues at $s=1/2$ and $s=\rho$. Now we observe that
\[
-\frac{1}{2\pi i}\int_{\varepsilon-iT}^{\varepsilon+iT}\frac{\zeta\left(2s\right)}{\zeta\left(s\right)}\frac{x^{s}}{s}ds=I(x)-\frac{1}{\pi i}\text{Re}\left(\int_{\varepsilon+iT}^{\varepsilon+i\infty}\frac{\zeta\left(2s\right)}{\zeta\left(s\right)}\frac{x^{s}}{s}ds\right)
\]
and
\[
-\frac{1}{2\pi i}\left(\int_{\sigma_{0}+iT}^{\varepsilon+iT}+\int_{\varepsilon-iT}^{\sigma_{0}-iT}\right)\frac{\zeta\left(2s\right)}{\zeta\left(s\right)}\frac{x^{s}}{s}ds=\frac{1}{\pi}\text{Im}\left(\int_{\varepsilon+iT}^{\sigma_{0}+iT}\frac{\zeta\left(2s\right)}{\zeta\left(s\right)}\frac{x^{s}}{s}ds\right).
\]
Then, using the estimations (\ref{eq:first est})
and (\ref{eq:second est}) with $\text{\ensuremath{\varepsilon/3}}$
and setting $s=\sigma+it$, we obtain
\[
\frac{1}{\pi i}\text{Re}\left(\int_{\varepsilon+iT}^{\varepsilon+i\infty}\frac{\zeta\left(2s\right)}{\zeta\left(s\right)}\frac{x^{s}}{s}ds\right)\ll x^{\varepsilon}\int_{T}^{+\infty}t^{-1-\varepsilon/3}dt=\frac{x^{\varepsilon}}{T^{\varepsilon/3}}\ll1.
\]
In a similar manner, using (\ref{eq:special est}) with $\varepsilon/2$,
\[
\frac{1}{\pi}\text{Im}\left(\int_{\varepsilon+iT}^{\sigma_{0}+iT}\frac{\zeta\left(2s\right)}{\zeta\left(s\right)}\frac{x^{s}}{s}ds\right)\ll T^{-1/2+\varepsilon}\int_{\varepsilon}^{1/4}T^{-2\sigma}x^{\sigma}d\sigma+T^{-1+\varepsilon}\int_{1/4}^{\sigma_{0}}x^{\sigma}d\sigma
\]
\[
\ll\frac{x^{\varepsilon}+1}{T^{1/2}}+\frac{x^{\sigma_{0}}-x^{1/4}}{T^{1-\varepsilon}\log\left(x\right)}
\]
and this completes the proof. 
\end{proof}
If we replace the function $\zeta\left(2s\right)$ with the constant
function $1$ then we have a similar result. Indeed, we have
\begin{thm}
Assume RH and SZ. Then for $T_{\nu}\in\mathcal{T}$ and $x>0$ we
get
\[
M(x)=\sum_{\left|\gamma\right|<T_{\nu}}\frac{x^{\rho}}{\zeta^{\prime}\left(\rho\right)\rho}+E^{*}(x)+I^{*}(x)+R^{*}\left(x,T_\nu\right)
\]
where
\[
E^{*}(x):=\begin{cases}
\mu\left(x\right)/2, & x\in\mathbb{N}\\
0, & \text{otherwise,}
\end{cases}
\]

\[
I^{*}\left(x\right):=\frac{1}{2\pi i}\int_{\left(\varepsilon\right)}\frac{1}{\zeta\left(s\right)}\frac{x^{s}}{s}ds
\]
and
\[
R^{*}\left(x,T_\nu\right)\ll_{\varepsilon}1+\frac{x\left(\left|\log\left(x\right)\right|+1\right)}{T_{\nu}}+\frac{x^{\varepsilon}}{T_{\nu}^{1/2-2\varepsilon}}+\frac{x-x^{\varepsilon}}{T_{\nu}^{1-\varepsilon}\log\left(x\right)}
\]
for arbitrarily small $0<\varepsilon<1/2$, where the indeterminate form at $x=1$ must be interpreted as limit.
\end{thm}

\begin{proof}
The proof is very similar to the previous one. Take $x\geq1$, then
we have just
\[
\frac{1}{\pi i}\text{Re}\left(\int_{\varepsilon+iT}^{\varepsilon+i\infty}\frac{1}{\zeta\left(s\right)}\frac{x^{s}}{s}ds\right)\ll x^{\varepsilon}\int_{T}^{+\infty}t^{-3/2+2\varepsilon}dt\ll\frac{x^{\varepsilon}}{T^{1/2-2\varepsilon}}
\]
and if $0<x<1$ the same estimation gives
\[
\frac{1}{\pi i}\text{Re}\left(\int_{\varepsilon+iT}^{\varepsilon+i\infty}\frac{1}{\zeta\left(s\right)}\frac{x^{s}}{s}ds\right)\ll1.
\]
Then, assuming $T\in\mathcal{T}$, we need to consider
\[
\frac{1}{\pi}\text{Im}\left(\int_{\varepsilon+iT}^{\sigma_{0}+iT}\frac{1}{\zeta\left(s\right)}\frac{x^{s}}{s}ds\right)\ll T^{-1+\varepsilon}\int_{\varepsilon}^{\sigma_{0}}x^{\sigma}d\sigma\ll \frac{x^{\sigma_{0}}-x^{\varepsilon}}{T^{1-\varepsilon}\log\left(x\right)}
\]
and note that the previous estimation holds for every $x>0$.
\end{proof}
Note that $I(x)$ is bounded by a constant for every $x>0$ (see \cite{F})
hence it can be included in the error term. Concerning $I^{*}(x)$, it
is enough to see, for our aims, that the integral is trivially $O(x^{\varepsilon})$,
even if it is well-known that a better approximation can be made,
due to the fact that we can ``push'' the real part of the line of
integration to $-\infty$ if $x>1$ (if $0<x\leq1$ the trivial approximation
gives immediately $O(1)$).

So we finally get the following Corollary, that allows us to remove
the condition $T\in\mathcal{T}.$
\begin{cor}
\label{cor:IMP}Assume RH and SZ. Then for all $T\geq1$ and $x>0$
we have
\[
L(x)=\frac{x^{1/2}}{\zeta\left(\frac{1}{2}\right)}+\sum_{\rho:\left|\gamma\right|<T}\frac{\zeta\left(2\rho\right)x^{\rho}}{\zeta^{\prime}\left(\rho\right)\rho}+R\left(x,T\right)
\]
where
\[
R\left(x,T\right)\ll1+\frac{x\left(\left|\log\left(x\right)\right|+1\right)}{T}+\frac{x-x^{1/4}}{T^{1-\varepsilon}\log\left(x\right)}.
\]
\end{cor}

\begin{proof}
The proof is essentially the same of \cite{H}, Corollary $4.9$.
Indeed, if $\nu\geq1$, taking $\nu\leq T_{\nu}\leq T\leq\nu+1$ we
have, for all $0<\alpha\leq1$, that
\[
\frac{1}{T_{\nu}^{\alpha}}-\frac{1}{T^{\alpha}}\leq\frac{\left(\nu+1\right)^{\alpha}-\nu^{\alpha}}{\nu^{\alpha}\left(\nu+1\right)^{\alpha}}\leq\frac{1}{T^{\alpha}}
\]
and if $\nu\leq T<T_{\nu}\leq\nu+1$ we have similarly that
\[
\frac{1}{T^{\alpha}}-\frac{1}{T_{\nu}^{\alpha}}\leq\frac{\left(\nu+1\right)^{\alpha}-\nu^{\alpha}}{\nu^{\alpha}\left(\nu+1\right)^{\alpha}}\leq\frac{1}{T^{\alpha}}
\]
 and then we can replace $R\left(x,T_{\nu}\right)$ with $R\left(x,T\right).$
Moreover, following \cite{H}, we observe that if $0<x\leq1$ we get
\[
\left|\sum_{\rho:T_{\nu}\leq\left|\gamma\right|<T}\frac{\zeta\left(2\rho\right)x^{\rho}}{\zeta^{\prime}\left(\rho\right)\rho}\right|\ll x^{1/2}\frac{\log^{1/2}\left(T\right)\log\left(\log\left(T\right)\right)}{\sqrt{T}}\ll1
\]
and hence the thesis.
\end{proof}
Very similar computations lead to the following corollary.
\begin{cor}
\label{cor:IMP-1}Assume RH and SZ. Then for all $T\geq1$, $\varepsilon>0$
and $x>0$ we have
\[
M(x)=\sum_{\rho:\left|\gamma\right|<T}\frac{x^{\rho}}{\zeta^{\prime}\left(\rho\right)\rho}+R^{*}\left(x,T\right)
\]
where
\[
R^{*}\left(x,T\right)\ll1+x^{\varepsilon}+\frac{x\left(\left|\log\left(x\right)\right|+1\right)}{T}+\frac{x-x^{\varepsilon}}{T^{1-\varepsilon}\log\left(x\right)}.
\]
\end{cor}

\section{convolution of the summatory function \texorpdfstring{$L(x)$}{L(x)}}

In this part we prove the explicit formula for the Laplace convolution
of $L(x)$ with itself. We recall that this is also the Ces\`aro
average of order $1$ of the function $S(n)$. It is important to
remark that, in the error term, we have to consider both the cases
when $x$ is near to $0$ and when it goes to infinity. Note that if
$0<x<1$ the error $O_{\varepsilon}\left(x\right)$ incorporates the
main terms, but this is not problem for our aims.
\begin{thm}
\label{thm:convol}Assume RH and SZ. Then, for any $x>0$ and sufficiently
small $\varepsilon>0$ we have
\[
  C_\lambda(x)
  =
\sum_{n\leq x}S\left(n\right)\left(x-n\right)=\int_{0}^{x}L\left(y\right)L\left(x-y\right)dy
\]
\[
=\frac{x^{2}\pi}{8\zeta\left(\frac{1}{2}\right)^{2}}+\frac{\sqrt{\pi}}{\zeta\left(\frac{1}{2}\right)}\sum_{\rho}\frac{\zeta\left(2\rho\right)\Gamma\left(\rho\right)x^{\rho+3/2}}{\zeta^{\prime}\left(\rho\right)\Gamma\left(\rho+\frac{5}{2}\right)}
\]
\[
+\sum_{\rho_{1}}\frac{\zeta\left(2\rho_{1}\right)}{\zeta^{\prime}\left(\rho_{1}\right)}\sum_{\rho_{2}}\frac{\zeta\left(2\rho_{2}\right)x^{\rho_{1}+\rho_{2}+1}}{\zeta^{\prime}\left(\rho_{2}\right)}\frac{\Gamma\left(\rho_{1}\right)\Gamma\left(\rho_{2}\right)}{\Gamma\left(\rho_{1}+\rho_{2}+2\right)}+O_{\varepsilon}\left(x^{3/2+\varepsilon}+x\right)
\]
where the single series and the double series over the non-trivial
zeros are absolutely convergent.
\end{thm}

\begin{proof}
First of all we note that, using Corollary \ref{cor:IMP}, we have
\[
\int_{0}^{x}L\left(y\right)L\left(x-y\right)dy
\]
\[
=\int_{0}^{x}\left(\frac{y^{1/2}}{\zeta\left(\frac{1}{2}\right)}+\sum_{\rho:\left|\gamma\right|<T}\frac{\zeta\left(2\rho\right)y^{\rho}}{\zeta^{\prime}\left(\rho\right)\rho}\right)\left(\frac{(x-y)^{1/2}}{\zeta\left(\frac{1}{2}\right)}+\sum_{\rho:\left|\gamma\right|<T}\frac{\zeta\left(2\rho\right)(x-y)^{\rho}}{\zeta^{\prime}\left(\rho\right)\rho}\right)dy
\]
\[
+2\int_{0}^{x}L\left(y\right)R\left(x-y,T\right)dy-\int_{0}^{x}R\left(y\right)R\left(x-y\right)dy.
\]
\[
=J_{1}+J_{2}+J_{3}
\]
say. The integration of $J_{1}$ is trivial since all the products
lead, essentially, to Beta functions. So we have
\[
J_{1}=\frac{x^{2}\pi}{8\zeta\left(\frac{1}{2}\right)^{2}}+\frac{\sqrt{\pi}}{\zeta\left(\frac{1}{2}\right)}\sum_{\rho:\left|\gamma\right|<T}\frac{\zeta\left(2\rho\right)\Gamma\left(\rho\right)x^{\rho+3/2}}{\zeta^{\prime}\left(\rho\right)\Gamma\left(\rho+\frac{5}{2}\right)}
\]
\[
+\sum_{\rho_{1}:\left|\gamma_{1}\right|<T}\frac{\zeta\left(2\rho_{1}\right)}{\zeta^{\prime}\left(\rho_{1}\right)}\sum_{\rho_{2}:\left|\gamma_{2}\right|<T}\frac{\zeta\left(2\rho_{2}\right)x^{\rho_{1}+\rho_{2}+1}}{\zeta^{\prime}\left(\rho_{2}\right)}\frac{\Gamma\left(\rho_{1}\right)\Gamma\left(\rho_{2}\right)}{\Gamma\left(\rho_{1}+\rho_{2}+2\right)}.
\]
Now we study $J_{2}$. It is well-known, under RH, that
\[
L\left(x\right)\ll x^{1/2+\varepsilon},\,\forall\varepsilon>0
\]
as $x\rightarrow+\infty$, so we can conclude that for every $\varepsilon>0$
and for every $x>0$ under RH we have the estimation
\[
L\left(x\right)\ll x^{1/2+\varepsilon}+1.
\]

From Corollary \ref{cor:IMP} and from the fact that $$ \frac{x-x^{1/4}}{\log\left(x\right)}\leq x+1$$ for every $x>0,\,x\neq1$, we have
\[
\int_{0}^{x}L\left(y\right)R\left(x-y,T\right)dy
\]
\[
\ll\int_{0}^{x}y^{1/2+\varepsilon}\left[1+\frac{(x-y)\left(\left|\log\left(x-y\right)\right|+1\right)}{T}+\frac{x-y+1}{T^{1-\varepsilon}}\right]dy
\]
\[
+\int_{0}^{x}\left[1+\frac{(x-y)\left(\left|\log\left(x-y\right)\right|+1\right)}{T}+\frac{x-y+1}{T^{1-\varepsilon}}\right]dy
\]
\[
=J_{21}+J_{22}.
\]
\[
\int_{0}^{x}y^{1/2+\varepsilon}dy\ll_{\varepsilon}x^{3/2+\varepsilon},\,\,\frac{1}{T}\int_{0}^{x}y^{1/2+\varepsilon}(x-y)dy=\frac{x^{5/2+\varepsilon}}{T}B\left(\frac{3}{2}+\varepsilon,2\right)
\]
so we have, essentially, just to consider
\[
\frac{1}{T}\int_{0}^{x}y^{1/2+\varepsilon}(x-y)\left|\log(x-y)\right|dy=\frac{1}{T}\int_{0}^{x}\left(x-y\right)^{1/2+\varepsilon}y\left|\log(y)\right|dy
\]
Note that, if $0<v\leq1$, then
\[
v\left|\log\left(v\right)\right|=v\log\left(1/v\right)<1
\]
hence if $0<x\leq1$
\[
\frac{1}{T}\int_{0}^{x}\left(x-y\right)^{1/2+\varepsilon}y\left|\log(y)\right|dy\leq\frac{1}{T}\int_{0}^{x}\left(x-y\right)^{1/2+\varepsilon}dy=\frac{x^{3/2+\varepsilon}}{T}.
\]
If $x>1$, we have
\[
\frac{1}{T}\int_{0}^{1}\left(x-y\right)^{1/2+\varepsilon}y\left|\log(y)\right|dy\ll\frac{x^{1/2+\varepsilon}}{T}
\]
and
\[
\frac{1}{T}\int_{1}^{x}\left(x-y\right)^{1/2+\varepsilon}y\log(y)dy\leq\log\left(x\right)\frac{1}{T}\int_{0}^{x}\left(x-y\right)^{1/2+\varepsilon}ydy
\]
\[
\ll\frac{\log\left(x\right)x^{5/2+\varepsilon}}{T}.
\]
So we can conclude that
\[
J_{21}\ll_{\varepsilon}x^{3/2+\varepsilon}+\frac{\left(\log\left(x\right)x^{5/2+\varepsilon}\right)_{x>1}+x_{0<x\leq1}^{3/2+\varepsilon}}{T}+\frac{x^{3/2+\varepsilon}+x^{5/2+\varepsilon}}{T^{1-\varepsilon}}
\]
and in a similar way
\[
J_{22}\ll x+\frac{x^{2}\log\left(x\right)_{x>1}+1_{0<x\leq1}}{T}+\frac{x+x^{2}}{T^{1-\varepsilon}}.
\]
It remains to consider $J_{3}.$ We have
\[
J_{3}\ll\int_{0}^{x}\left(1+\frac{y\left(\left|\log\left(y\right)\right|+1\right)}{T}+\frac{y+1}{T^{1-\varepsilon}}\right)
\]
\[
\times\left(1+\frac{\left(x-y\right)\left(\left|\log\left(x-y\right)\right|+1\right)}{T}+\frac{x-y+1}{T^{1-\varepsilon}}\right)dy.
\]
If we define
\[
A_{T}\left(y\right)=C_{T}\left(y\right):=1+\frac{1}{T^{1-\varepsilon}}\ll1
\]
\[
B_{T}\left(y\right):=\frac{y\left(\left|\log\left(y\right)\right|+1\right)}{T}+\frac{y}{T^{1-\varepsilon}},D_{T}\left(y\right):=\frac{\left(x-y\right)\left(\left|\log\left(x-y\right)\right|+1\right)}{T}+\frac{x-y}{T^{1-\varepsilon}}
\]
we get
\begin{equation}
J_{3}\ll x+\int_{0}^{x}D_{T}\left(y\right)dy+\int_{0}^{x}B_{T}\left(y\right)dy+\int_{0}^{x}B_{T}\left(y\right)D_{T}\left(y\right)dy\label{eq:eq ABCD}
\end{equation}
and it is very simple to observe that all the integrals in (\ref{eq:eq ABCD})
are convergent for every $x>0$ and they tend to $0$ as $T\rightarrow+\infty$,
hence the main formula follows.

It remains to observe that, actually, the series converges absolutely.
Indeed, for the first sum, by the classical Stirling formula we
have
\[
\sum_{\rho:\gamma>0}\left|\frac{\zeta\left(2\rho\right)\Gamma\left(\rho\right)}{\zeta^{\prime}\left(\rho\right)\Gamma\left(\rho+\frac{5}{2}\right)}\right|\ll\sum_{\rho:\gamma>0}\left|\frac{\zeta\left(2\rho\right)}{\zeta^{\prime}\left(\rho\right)}\right|\frac{1}{\gamma^{5/2}}
\]
and such series is convergent since, by SZ, we have
\[
\sum_{\rho:0<\gamma<T}\left|\frac{\zeta\left(2\rho\right)}{\zeta^{\prime}\left(\rho\right)}\right|\frac{1}{\gamma^{5/2}}\ll\sum_{\rho:0<\gamma<T}\left|\frac{\zeta\left(2\rho\right)}{\rho\zeta^{\prime}\left(\rho\right)}\right|\frac{1}{\gamma^{3/2}}
\]
\[
\ll\frac{\left(\log T\right)^{3/2}\log\left(\log\left(T\right)\right)}{T^{3/2}}+\int_{14}^{T}\left(\left(\log t\right)^{3/2}\log\left(\log\left(t\right)\right)\right)t^{-5/2}dt
\]
and the last integral is convergent if $T\rightarrow+\infty$. 

Clearly, by symmetry, the considered case $\gamma>0$ is enough. For
the double series the convergence follows from Theorem \ref{thm:double ser}.
\end{proof}
Again, it is quite simple to prove a similar result for the $M(x)$
function.
\begin{thm}
\label{thm:convol-1}Assume RH and SZ. Then, for any $x>0$ and sufficiently
small $\varepsilon>0$ we have
\[
  C_\mu(x)
  =
\sum_{n\leq x}S^{*}\left(n\right)\left(x-n\right)=\int_{0}^{x}M\left(y\right)M\left(x-y\right)dy
\]
\[
=\sum_{\rho_{1}}\sum_{\rho_{2}}\frac{x^{\rho_{1}+\rho_{2}+1}}{\zeta^{\prime}\left(\rho_{1}\right)\zeta^{\prime}\left(\rho_{2}\right)}\frac{\Gamma\left(\rho_{1}\right)\Gamma\left(\rho_{2}\right)}{\Gamma\left(\rho_{1}+\rho_{2}+2\right)}+O_{\varepsilon}\left(x^{3/2+\varepsilon}+x^{\varepsilon}\right)
\]
where the double series over the non-trivial zeros is absolutely convergent.
\end{thm}

\begin{proof}
The proof is essentially the same as Theorem~\ref{thm:convol}, since $M(x)\ll x^{1/2+\varepsilon}$
under RH, and the absolute convergence of the series follows from
Theorem \ref{thm:double ser}.
\end{proof}

\section{the weighted averages}

In this section we want to apply the general result present in \cite{CGZ}, Proposition 2, for the function $S(n).$ We start recalling such main theorem.
\begin{thm}
\label{thm:MAINCGZ}Let $g_{1},g_{2}$ be arithmetical functions,
$\eta>0$ $f:\mathbb{R}\rightarrow\mathbb{C}$, and assume that:

1) $f$ has its support in $\left[a,b\right),\,a<b$, and $b\in\mathbb{R}\cup\left\{ +\infty\right\} $

2) $f\in C^{1}\left(a,b\right)$.

3) $f^{\prime}$ is absolutely continuous in $\left(a,b\right)$.

4) $f\left(a^{+}\right),f^{\prime}\left(a^{+}\right)$ exist and are
finite, and $f\left(b^{-}\right)=f^{\prime}\left(b^{-}\right)=0$.

Then, if 
\[
G_{j}\left(x\right):=\begin{cases}
\sum_{n\leq x}g_{j}\left(x\right), & x>0\\
0, & \text{otherwise}
\end{cases}
\]
for $j=1,2$, we get
\[
\sum_{\eta a<n\leq\eta b}\sum_{m\leq\eta b-n}g_{2}\left(m\right)g_{1}\left(n\right)f\left(\frac{m+n}{\eta}\right)=G_{2}\left(\eta a\right)\int_{a}^{b}G_{1}\left(\eta v-\eta a\right)f^{\prime}\left(v\right)dv
\]
\[
+\frac{1}{\eta}\int_{a}^{b}f^{\prime\prime}\left(w\right)\int_{\eta a}^{\eta w}G_{2}\left(s\right)G_{1}\left(\eta w-s\right)dsdw.
\]
\end{thm}

So, taking $g_{1}(n)=g_{2}(n)=\lambda(n)$ and $G_{1}(x)=G_{2}(x)=L(x)$,
and assuming that $\eta a<1$, which implies that
\[
L\left(\eta a\right)=0,\,\int_{\eta a}^{\eta w}L\left(s\right)L\left(\eta w-s\right)ds=\left(L*L\right)\left(w\eta\right)
\]
the following result holds
\[
\sum_{n\leq\eta b}\sum_{m\leq\eta b-n}\lambda\left(n\right)\lambda\left(m\right)f\left(\frac{m+n}{\eta}\right)=\frac{\pi\eta}{8\zeta\left(\frac{1}{2}\right)^{2}}\int_{a}^{b}f^{\prime\prime}\left(w\right)w^{2}dw
\]
\[
+\frac{\sqrt{\pi}}{\zeta\left(\frac{1}{2}\right)}\sum_{\rho}\frac{\zeta\left(2\rho\right)\Gamma\left(\rho\right)\eta^{\rho+1/2}}{\zeta^{\prime}\left(\rho\right)\Gamma\left(\rho+\frac{5}{2}\right)}\int_{a}^{b}f^{\prime\prime}\left(w\right)w^{\rho+1/2+1}dw
\]
\[
+\sum_{\rho_{1}}\frac{\zeta\left(2\rho_{1}\right)}{\zeta^{\prime}\left(\rho_{1}\right)}\sum_{\rho_{2}}\frac{\zeta\left(2\rho_{2}\right)\eta^{\rho_{1}+\rho_{2}}}{\zeta^{\prime}\left(\rho_{2}\right)}\frac{\Gamma\left(\rho_{1}\right)\Gamma\left(\rho_{2}\right)}{\Gamma\left(\rho_{1}+\rho_{2}+2\right)}\int_{a}^{b}f^{\prime\prime}\left(w\right)w^{\rho_{1}+\rho_{2}+1}dw
\]
\[
+O_{\varepsilon}\left(\eta^{1/2+\varepsilon}\int_{a}^{b}w^{3/2+\varepsilon}\left|f^{\prime\prime}\left(w\right)\right|dw+\int_{a}^{b}w\left|f^{\prime\prime}\left(w\right)\right|dw\right)
\]
and integral and series can be switched by the absolute convergence
of the series. Now, if we fix $\eta>0$ and we take $f:\left[\frac{1}{\eta},+\infty\right)\rightarrow\mathbb{C},$
$f\left(w\right):=w^{-s}$, $\text{Re}(s)>2$, Now we are able to prove
what follows.
\begin{thm}
Assume RH and SZ. We have
\[
\sum_{n\geq1}\frac{S\left(n\right)}{n^{s}}=\frac{s\left(s+1\right)\pi}{8\zeta\left(\frac{1}{2}\right)^{2}\left(1-s\right)}+\frac{\sqrt{\pi}s\left(s+1\right)}{\zeta\left(\frac{1}{2}\right)}\sum_{\rho}\frac{\zeta\left(2\rho\right)\Gamma\left(\rho\right)}{\zeta^{\prime}\left(\rho\right)\Gamma\left(\rho+\frac{5}{2}\right)\left(\rho-s+\frac{1}{2}\right)}
\]
\[
+s\left(s+1\right)\sum_{\rho_{1}}\frac{\zeta\left(2\rho_{1}\right)}{\zeta^{\prime}\left(\rho_{1}\right)}\sum_{\rho_{2}}\frac{\zeta\left(2\rho_{2}\right)}{\zeta^{\prime}\left(\rho_{2}\right)}\frac{\Gamma\left(\rho_{1}\right)\Gamma\left(\rho_{2}\right)}{\Gamma\left(\rho_{1}+\rho_{2}+2\right)\left(\rho_{1}+\rho_{2}-s\right)}
\]
\[
+O\left(\frac{\left|s\left(s+1\right)\right|}{\text{Re}(s)-1/2-\varepsilon}\right)
\]
which is valid for $\text{Re}(s)>1$.
\end{thm}

%\footnote{Si pu\`o calcolare il valore degli ultimi due integrali}

\begin{proof}
We fix $\eta$ and we take $f:\left[\frac{1}{\eta},+\infty\right)\rightarrow\mathbb{C}$,
defined as $f\left(w\right):=w^{-s},\,\text{Re}\left(s\right)>2$.
In this case, we have, from Theorem \ref{thm:MAINCGZ} and from the
trivial evaluation $L(1)=1$, that
\[
\eta^{s}\sum_{n\geq1}\sum_{m\geq1}\frac{\lambda\left(n\right)\lambda\left(m\right)}{\left(m+n\right)^{s}}=\eta^{s}\sum_{n\geq1}\frac{S\left(n\right)}{n^{s}}=-s\int_{1/\eta}^{+\infty}L\left(\eta v-1\right)v^{-s-1}dv
\]
\[
+\frac{s\left(s+1\right)}{\eta}\int_{1/\eta}^{+\infty}w^{-s-2}\int_{1}^{\eta w}L\left(s\right)L\left(\eta w-s\right)dsdw.
\]
Now, if we make the change of variable $\eta v-1=g$ in the first
integral and $\eta w=h$ in the second, we get
\[
-s\int_{1/\eta}^{+\infty}L\left(\eta v-1\right)v^{-s-1}dv=-s\eta^{s}\int_{0}^{+\infty}\frac{L(g)}{\left(g+1\right)^{s+1}}dg
\]
and
\[
\frac{s\left(s+1\right)}{\eta}\int_{1/\eta}^{+\infty}w^{-s-2}\int_{1}^{\eta w}L\left(s\right)L\left(\eta w-s\right)dsdw
\]
\[
=s\left(s+1\right)\eta^{s}\int_{1}^{+\infty}w^{-s-2}\int_{1}^{h}L\left(s\right)L\left(h-s\right)dsdh
\]
\[
=s\left(s+1\right)\eta^{s}\int_{1}^{+\infty}w^{-s-2}\left(L*L\right)\left(h\right)dh.
\]
Hence, canceling out $\eta^{s}$, inserting the explicit formula of $\left(L*L\right)$
and making trivial computations, we get, for every $\varepsilon>0$,
that
\[
s\left(s+1\right)\int_{1}^{+\infty}w^{-s-2}\left(L*L\right)\left(h\right)dh=\frac{s\left(s+1\right)\pi}{8\zeta\left(\frac{1}{2}\right)^{2}\left(1-s\right)}
\]
\[
+\frac{\sqrt{\pi}s\left(s+1\right)}{\zeta\left(\frac{1}{2}\right)}\sum_{\rho}\frac{\zeta\left(2\rho\right)\Gamma\left(\rho\right)}{\zeta^{\prime}\left(\rho\right)\Gamma\left(\rho+\frac{5}{2}\right)\left(\rho-s+\frac{1}{2}\right)}
\]
\[
+s\left(s+1\right)\sum_{\rho_{1}}\frac{\zeta\left(2\rho_{1}\right)}{\zeta^{\prime}\left(\rho_{1}\right)}\sum_{\rho_{2}}\frac{\zeta\left(2\rho_{2}\right)}{\zeta^{\prime}\left(\rho_{2}\right)}\frac{\Gamma\left(\rho_{1}\right)\Gamma\left(\rho_{2}\right)}{\Gamma\left(\rho_{1}+\rho_{2}+2\right)\left(\rho_{1}+\rho_{2}-s\right)}
\]
\[
+O\left(\left|s\left(s+1\right)\right|\left[\int_{1}^{+\infty}w^{-\text{Re}(s)-1/2+\varepsilon}dw+\int_{1}^{+\infty}w^{-\text{Re}(s)-2}dw\right]\right).
\]
Clearly, if $\text{Re}(s)>1$, all the integrals are convergent, the main terms of the formula are well defined and the series are absolutely convergent.
To conclude it is enough to observe that, for every $\varepsilon>0$,
we get
\[
-s\int_{0}^{+\infty}\frac{L(g)}{\left(g+1\right)^{s+1}}dg\ll\left|s\right|\int_{0}^{+\infty}\frac{1}{\left(g+1\right)^{\text{Re}(s)+1/2-\varepsilon}}dg
\]
which can be included in the error term since $\left|s+1\right|>1.$
\end{proof}
Note that the previous corollary provides an analytic continuation
of the Dirichlet series to the half-plane $\text{Re}(s)>1.$ We cannot
expect to overcome the line $\text{Re}\left(s\right)=1$ since it
is well-known that, under RH and under the hypothesis that the imaginary
parts of the non-trivial zeros of $\zeta(s)$ are linearly independent
over the rationals, then the set
\[
\mathcal{K}:=\left\{ k\in\mathbb{R}:\,k=\gamma_{1}+\gamma_{2},\,\gamma_{1},\gamma_{2}\text{ are imaginary parts of the non-trivial zeros of }\zeta\left(s\right)\right\} 
\]
is dense in $\mathbb{R}$ (see \cite{EM} for the Goldbach case, which
presents the same situation). 

In a similar fashion, fixing two parameters $\eta,\,y>0$, taking
$f:\mathbb{R}_{0}^{+}\rightarrow\mathbb{C},$ $f\left(w\right):=e^{-w\eta y},$
we get the following explicit formula
\begin{align*}
\sum_{n\geq1}\sum_{m\geq1}\lambda\left(n\right)\lambda\left(m\right)f\left(\frac{m+n}{\eta}\right) & =\sum_{n\geq1}S\left(n\right)e^{-ny}\\
 & =\frac{\pi}{4\zeta\left(\frac{1}{2}\right)^{2}y}+\frac{\sqrt{\pi}}{\zeta\left(\frac{1}{2}\right)}\sum_{\rho}\frac{\zeta\left(2\rho\right)}{\zeta^{\prime}\left(\rho\right)}\Gamma\left(\rho\right)y^{-\rho-1/2}\\
 & +\sum_{\rho_{1}}\frac{\zeta\left(2\rho_{1}\right)}{\zeta^{\prime}\left(\rho_{1}\right)}\Gamma\left(\rho_{1}\right)y^{-\rho_{1}}\sum_{\rho_{2}}\frac{\zeta\left(2\rho_{2}\right)\Gamma\left(\rho_{2}\right)}{\zeta^{\prime}\left(\rho_{2}\right)}y^{-\rho_{2}}\\
 & +O_{\varepsilon}\left(y^{-1/2-\varepsilon}+1\right).
\end{align*}

Again, an identical argument produces similar formulas in the
case we work with $\mu(n)$. For example, using the notations and
the hypotheses of Theorem \ref{thm:MAINCGZ} and assuming that $a\eta<1$,
we have, for every $\varepsilon>0$, that
\[
\sum_{n\leq\eta b}\sum_{m\leq\eta b-n}\mu\left(n\right)\mu\left(m\right)f\left(\frac{m+n}{\eta}\right)
\]
\[
=\sum_{\rho_{1}}\sum_{\rho_{2}}\frac{\eta^{\rho_{1}+\rho_{2}}}{\zeta^{\prime}\left(\rho_{1}\right)\zeta^{\prime}\left(\rho_{2}\right)}\frac{\Gamma\left(\rho_{1}\right)\Gamma\left(\rho_{2}\right)}{\Gamma\left(\rho_{1}+\rho_{2}+2\right)}\int_{0}^{b}f^{\prime\prime}\left(w\right)w^{\rho_{1}+\rho_{2}+1}dw
\]

\[
+O_{\varepsilon}\left(\eta^{1/2+\varepsilon}\int_{0}^{b}w^{3/2+\varepsilon}\left|f^{\prime\prime}\left(w\right)\right|dw+\int_{0}^{b}w\left|f^{\prime\prime}\left(w\right)\right|dw\right).
\]
Furthermore, we can get the following results.
\begin{cor}
Assume RH and SZ. We have, for every $\varepsilon>0$, that
\[
\sum_{n\geq1}\frac{S^{*}\left(n\right)}{n^{s}}=s\left(s+1\right)\sum_{\rho_{1}}\frac{1}{\zeta^{\prime}\left(\rho_{1}\right)}\sum_{\rho_{2}}\frac{1}{\zeta^{\prime}\left(\rho_{2}\right)}\frac{\Gamma\left(\rho_{1}\right)\Gamma\left(\rho_{2}\right)}{\Gamma\left(\rho_{1}+\rho_{2}+2\right)\left(\rho_{1}+\rho_{2}-s\right)}
\]
\[
+O\left(\frac{\left|s\left(s+1\right)\right|}{\text{Re}(s)-1/2-\varepsilon}\right)
\]
which is valid for $\text{Re}(s)>1$.
\end{cor}

\begin{cor}
Assume RH and SZ. For $y>0$ and for every $\varepsilon>0$ we have
\[
\sum_{n\geq1}S^{*}\left(n\right)e^{-ny}=\sum_{\rho_{1}}\frac{\Gamma\left(\rho_{1}\right)}{\zeta^{\prime}\left(\rho_{1}\right)}y^{-\rho_{1}}\sum_{\rho_{2}}\frac{\Gamma\left(\rho_{2}\right)}{\zeta^{\prime}\left(\rho_{2}\right)}y^{-\rho_{2}}+O_{\varepsilon}\left(y^{-1/2-\varepsilon}+1\right).
\]
\end{cor}

\section{liouville/m\"obius additive problem with an arbitrary number of
summands}

From the theory it is quite simple to generalize the previous results
to an additive problem with Liouville function with an arbitrary number
of summands. More precisely, let $d\geq2$ and consider
\[
S_{d}\left(n\right):=\sum_{m_{1}+\dots+m_{d}=n}\lambda\left(m_{1}\right)\cdots\lambda\left(m_{d}\right).
\]
From \cite{CGZ} we know that Theorem \ref{thm:MAINCGZ} can be generalized
in the following way
\[
\sum_{n_{1}\leq\eta b}\sum_{n_{2}\leq\eta b-n_{1}}\cdots\sum_{n_{d}\leq\eta b-n_{1}-\dots-n_{d-1}}g_{1}\left(n_{1}\right)\cdots g_{d}\left(n_{d}\right)f\left(\frac{n_{1}+\dots+n_{d}}{\eta}\right)
\]
\[
=\frac{1}{\eta}\int_{0}^{b}f^{(d)}\left(w\right)\left(G_{1}*\dots*G_{d}\right)\left(\eta w\right)dw
\]
and letting
\[
\mathcal{G}_{d}(n):=\sum_{m_{1}+\dots+m_{d}=n}g_{1}\left(m_{1}\right)\cdots g_{d}\left(m_{d}\right)
\]
we have
\[
\left(G_{1}*\dots*G_{d}\right)\left(x\right)=\frac{1}{\left(d-1\right)!}\sum_{n\leq x}\mathcal{G}_{d}\left(n\right)\left(x-n\right)^{d-1}=\frac{1}{\left(d-2\right)!}\int_{0}^{x}\sum_{n\leq y}\mathcal{G}_{d-1}\left(n\right)\left(y-n\right)^{d-2}dy
\]
\[
=\frac{1}{\left(d-3\right)!}\int_{0}^{x}\int_{0}^{y}\sum_{n\leq z}\mathcal{G}_{d-2}\left(n\right)\left(z-n\right)^{d-3}dzdy=\dots
\]
that is, we just need to integrate the Ces\`aro average of order 1
$d-2$ times (see again \cite{CGZ}). So we have the following theorem.
\begin{thm}
For every fixed natural number $d\geq2$ we have
\[
\left(L*\dots*L\right)\left(x\right)=\frac{1}{\left(d-1\right)!}\sum_{n\leq x}S_{d}\left(n\right)\left(x-n\right)^{d-1}
\]
\[
=\frac{x^{d}\pi}{4\zeta\left(\frac{1}{2}\right)^{2}d!}+\frac{\sqrt{\pi}}{\zeta\left(\frac{1}{2}\right)}\sum_{\rho}\frac{\zeta\left(2\rho\right)\Gamma\left(\rho\right)x^{\rho+d-1/2}}{\zeta^{\prime}\left(\rho\right)\Gamma\left(\rho+d+\frac{1}{2}\right)}
\]
\[
+\sum_{\rho_{1}}\frac{\zeta\left(2\rho_{1}\right)}{\zeta^{\prime}\left(\rho_{1}\right)}\sum_{\rho_{2}}\frac{\zeta\left(2\rho_{2}\right)x^{\rho_{1}+\rho_{2}+d-1}}{\zeta^{\prime}\left(\rho_{2}\right)}\frac{\Gamma\left(\rho_{1}\right)\Gamma\left(\rho_{2}\right)}{\Gamma\left(\rho_{1}+\rho_{2}+d\right)}+O_{\varepsilon}\left(x^{d-1/2+\varepsilon}+x^{d-1}\right).
\]
\end{thm}

We can conclude that if $f$ is a function like in the Corollary 4
of \cite{CGZ}, then we get
\[
\sum_{n_{1}\leq\eta b}\sum_{n_{2}\leq\eta b-n_{1}}\cdots\sum_{n_{d}\leq\eta b-n_{1}-\dots-n_{d-1}}\lambda\left(n_{1}\right)\cdots\lambda\left(n_{d}\right)f\left(\frac{n_{1}+\dots+n_{d}}{\eta}\right)
\]
\[
=\frac{\pi\eta^{d-1}}{4\zeta\left(\frac{1}{2}\right)^{2}d!}\int_{0}^{b}f^{\prime\prime}\left(w\right)w^{d}dw+\frac{\sqrt{\pi}}{\zeta\left(\frac{1}{2}\right)}\sum_{\rho}\frac{\zeta\left(2\rho\right)\Gamma\left(\rho\right)\eta^{\rho+d-3/2}}{\zeta^{\prime}\left(\rho\right)\Gamma\left(\rho+d+\frac{1}{2}\right)}\int_{0}^{b}f^{\prime\prime}\left(w\right)w^{\rho+d-1/2}dw
\]
\[
+\sum_{\rho_{1}}\frac{\zeta\left(2\rho_{1}\right)}{\zeta^{\prime}\left(\rho_{1}\right)}\sum_{\rho_{2}}\frac{\zeta\left(2\rho_{2}\right)\eta^{\rho_{1}+\rho_{2}+d-2}}{\zeta^{\prime}\left(\rho_{2}\right)}\frac{\Gamma\left(\rho_{1}\right)\Gamma\left(\rho_{2}\right)}{\Gamma\left(\rho_{1}+\rho_{2}+2\right)}\int_{0}^{b}f^{\prime\prime}\left(w\right)w^{\rho_{1}+\rho_{2}+d-1}dw
\]
\[
+O_{\varepsilon}\left(\eta^{d-3/2+\varepsilon}\int_{0}^{b}w^{d-1/2+\varepsilon}\left|f^{\prime\prime}\left(w\right)\right|dw+\eta^{d-2}\int_{0}^{b}w^{d-1}\left|f^{\prime\prime}\left(w\right)\right|dw\right).
\]

A very similar results can be obtained if we consider the M\"obius
function using the same argument, giving
\[
\sum_{n_{1}\leq\eta b}\sum_{n_{2}\leq\eta b-n_{1}}\cdots\sum_{n_{d}\leq\eta b-n_{1}-\dots-n_{d-1}}\mu\left(n_{1}\right)\cdots\mu\left(n_{d}\right)f\left(\frac{n_{1}+\dots+n_{d}}{\eta}\right)
\]
\[
=\sum_{\rho_{1}}\frac{1}{\zeta^{\prime}\left(\rho_{1}\right)}\sum_{\rho_{2}}\frac{\eta^{\rho_{1}+\rho_{2}+d-2}}{\zeta^{\prime}\left(\rho_{2}\right)}\frac{\Gamma\left(\rho_{1}\right)\Gamma\left(\rho_{2}\right)}{\Gamma\left(\rho_{1}+\rho_{2}+2\right)}\int_{0}^{b}f^{\prime\prime}\left(w\right)w^{\rho_{1}+\rho_{2}+d-1}dw
\]
\[
+O_{\varepsilon}\left(\eta^{d-3/2+\varepsilon}\int_{0}^{b}w^{d-1/2+\varepsilon}\left|f^{\prime\prime}\left(w\right)\right|dw+\eta^{d-2}\int_{0}^{b}w^{d-1}\left|f^{\prime\prime}\left(w\right)\right|dw\right).
\]

\section*{acknowledgements}
The first and the second author are members of the Gruppo Nazionale per l'Analisi Matematica, la
Probabilit\`a e le loro Applicazioni (GNAMPA) of the Istituto Nazionale di Alta Matematica (INdAM). The third
author is a member of the Gruppo Nazionale per le Strutture Algebriche, Geometriche e le loro Applicazioni
(GNSAGA) of the Istituto Nazionale di Alta Matematica (INdAM). We thank Professor Jasson Vindas for several discussions about this work.

$\ $

\begin{tabular}{l} Marco Cantarini \\ Dipartimento di Matematica e Informatica \\ Universit\`a di Perugia \\ Via Vanvitelli, 1 \\ 06123, Perugia, Italy \\ email (MC): \texttt{marco.cantarini@unipg.it} \\ 
$\,$
Alessandro Gambini\\ Dipartimento di Matematica\\ Universit\`a degli Studi di Bologna\\  Piazza di Porta San Donato, 5\\ 40126 Bologna, Italy\\ email (AG): \texttt{a.gambini@unibo.it} \\ 
$\,$
Alessandro Zaccagnini \\ Dipartimento di Scienze, Matematiche, Fisiche e Informatiche \\ Universit\`a di Parma \\ Parco Area delle Scienze 53/a \\ 43124 Parma, Italy \\ email (AZ): \texttt{alessandro.zaccagnini@unipr.it} \end{tabular}
\end{document}